\documentclass[a4paper,12pt,intlimits,oneside]{amsart}
\usepackage{graphicx}
\usepackage{enumerate}
\usepackage{amsfonts}
\usepackage{amsmath,amsthm,amssymb}
\usepackage{float}
\newcommand{\comment}[1]{}

\newtheorem{theorem}{Theorem}

\newtheorem{lemma}{Lemma}

\newcommand\supp{\qopname\relax o{supp}}

\newcommand{\RR}{\mathbb R}

\newcommand{\cH}{\mathcal{H}}

\newcommand{\Lh}{\mathcal{L}(\mathcal{H})}
\newcommand{\Lhs}{\mathcal{L}(\mathcal{H})_{sa}}
\newcommand{\Lhp}{\mathcal{L}(\mathcal{H})_+}
\newcommand{\Lhpi}{\mathcal{L}(\mathcal{H})_{++}}
\newcommand{\Ph}{\mathcal{P}_1(\cH)}
\newcommand\tr{\operatorname{Tr}}
\newcommand\sign{\operatorname{sign}}

\begin{document}
\title[R\'enyi type relative entropies and maximal $f$-divergences]{Maps on positive operators preserving R\'enyi type relative entropies and maximal $f$-divergences}
\author{Marcell Ga\'al and Gerg\H{o} Nagy}

\address{Bolyai Institute, University of Szeged \newline  \indent H-6720 Szeged, Aradi v\'ertan\'uk tere 1. \newline  \indent and \newline  \indent MTA-DE "Lend\"ulet" Functional Analysis Research Group \newline \indent Institute of Mathematics, University of Debrecen \newline \indent H-4002 Debrecen, PO. Box 400} \email{marcell.gaal.91@gmail.com}
\email{nagyg@science.unideb.hu}

\begin{abstract} In this paper we deal with two quantum relative entropy preserver problems on the cones of positive (either positive definite or positive semidefinite) operators. The first one is related to a quantum R\'enyi relative entropy like quantity which plays an important role in classical-quantum channel decoding. The second one is connected to the so-called maximal $f$-divergences introduced by D. Petz and M. B. Ruskai who considered this quantity as a generalization of the usual Belavkin-Staszewski relative entropy. We emphasize in advance that all the results are obtained for finite dimensional Hilbert spaces.
\end{abstract}

\maketitle

\bigskip
\bigskip



\bigskip
\bigskip
	
\section{Introduction}

Relative entropy is one of the most important numerical quantity in quantum information theory. It is used as a measure of distinguishability between quantum states, or more generally, positive (either positive semidefinite or positive definite) operators. However, from the purely mathematical point of view any metric would do this job, for certain problems in quantum information theory (such as e.g. error correction or channel decoding) it is more beneficial to deal with some particular relative entropies and divergences.

The original concept of relative entropy
was generalized in many different ways in the past decades. The most extensively studied generalizations are the \emph{R\'enyi type relative entropies} and different \emph{quantum $f$-divergences}. In principle, for positive definite operators\footnote{For the sake of simplicity in the Introduction we take the freedom to define the corresponding divergences only on the set of positive definite operators. The extensions for general (not necessarily invertible) positive semidefinite operators are discussed in Section 2.} $A,B$ three different families of R\'enyi divergences are defined as follows. Let $\alpha \in]0,1[\cup]1,+\infty[$ be an arbitrary but fixed parameter. Then the \emph{(conventional) R\'enyi divergence} is the quantity
\[
D_{\alpha}(A\Vert B)=(\alpha-1)^{-1} \log \left((\tr A)^{-1}\tr A^{\alpha}B^{1-\alpha}\right),
\]
and the so-called \emph{sandwiched R\'enyi relative entropy} \cite{Muletal13} is
\[
D_{\alpha}^{*}(A\Vert B)=(\alpha-1)^{-1} \log \left((\tr A)^{-1} \tr \left(B^{\frac{1-\alpha}{2\alpha}}AB^{\frac{1-\alpha}{2\alpha}}  \right)^{\alpha}\right).
\]
In \cite{mosonyi} M. Mosonyi discovered and studied extensively the following version of R\'enyi divergences:
\[
D_{\alpha}^{\flat}(A\Vert B)= (\alpha-1)^{-1} \log\left((\tr A)^{-1} \tr \exp(\alpha \log A+(1-\alpha)\log B)\right)
\]
which plays an important role in classical-quantum channel decoding.
The expression $$ Q_{\alpha}(A\Vert B)=(\tr A)^{-1} \tr A^{\alpha}B^{1-\alpha}$$
appearing in the definition of $D_{\alpha}(.\Vert .)$ is a \emph{standard $f$-divergence} \cite{HMPB} on the set of density operators  corresponding to the function $f_{\alpha}(t)=\sign(\alpha -1)t^{\alpha}$. Motivated by the celebrated Wigner's theorem concerning the structure of quantum mechanical symmetry transformations (i.e. maps on pure states preserving the transition probability) in \cite{MNSz} (for any strictly convex function $f$) the second author and his coauthors determined the structure of transformations on the set of density operators leaving the standard quantum $f$-divergences invariant. In \cite{virosztek} D. Virosztek managed to determine the corresponding preservers on the cone of positive semidefinite operators.
In \cite{gaal} the first author and L. Moln\'ar described the symmetry transformations with respect to the sandwiched R\'enyi relative entropy.
By following this line of research, our first aim here is to describe the quantum R\'enyi divergence preservers in the case where the quantity $D_{\alpha}^{\flat}(.\Vert .)$ is considered.

As for the second, in the paper \cite{ruskai} D. Petz and M. B. Ruskai introduced for any operator convex function $f$ the numerical quantity
\[
D_f(A\Vert B)=\tr Bf(B^{-1/2}AB^{-1/2})
\]
from which one can recover the usual Belavkin-Staszewski relative entropy by putting the so-called standard entropy function $\eta(t)=t\log t$ in the place of $f$. In \cite{Matsumoto} K. Matsumoto pointed out that the quantity $D_f(.\Vert .)$ is maximal among the monotone quantum $f$-divergences, and it can be expressed in the form of a natural optimization of the $f$-divergences of classical distribution functions that can be mapped into the given quantum operators (see also \cite{HM}). Hence, following the terminology introduced by K. Matsumoto, we refer to the quantities $D_f(.\Vert .)$ as \emph{maximal $f$-divergences}. In \cite{banach} the first author\footnote{On the set of positive definite operators the case where $f$ is operator monotone decreasing has been already solved before by L. Moln\'ar in \cite{Sebestyen}.} determined the structure of those transformations on the cones of positive operators leaving maximal $f$-divergences invariant. However, in that paper the case where $\lim_{t\to\infty}f(t)/t<+\infty$ has not been handled yet on the set of positive semidefinite operators. Therefore, the second aim of the present paper is to determine the corresponding preservers in the missing case and in that way complete our former work.

As we have already mentioned, the current paper is a contribution to the authors' former works but for further relative entropy preserver problems on quantum states or cones of positive operators the interested reader can consult with the series of publications \cite{gaal}, \cite{elso}, \cite{Sebestyen}, \cite{pitrik}, \cite{bregman}, \cite{virosztek}. In fact, at several points we adopt some ideas and techniques developed in some of these latter articles, especially the one that reduces our original preserver problem to the description of the form of order automorphisms with respect to the usual or the chaotic order.

This paper is organized as follows. In the second section we briefly review the basic notions and some necessary preliminaries which will be used in the sequel. In Section 3 we formulate our main results, and the proofs are presented separately in Section 4. Finally, in Section 5 we answer some questions which are closely related to the problems discussed in Sections 3-4.

\section{Preliminaries}

\subsection{Notation}

Let $\cH$ be a complex finite dimensional Hilbert space with $d=\dim \cH \geq 2$. The symbols $\Lh$ and $\Lhs$ stand for the algebra of all linear operators and the linear space of Hermitian operators on $\cH$, respectively. The usual L\"owner order among the elements of $\Lhs$ is denoted by $\leq$, we write $A\leq B$ if and only if $\langle A x,x\rangle \leq \langle B x,x\rangle$ holds for every $x \in \cH$. The operator $A\in\Lh$ is said to be positive semidefinite if $A \geq 0$. The cones of positive semidefinite operators and positive definite (i.e. invertible positive semidefinite) operators on $\cH$ are denoted by $\Lhp$ and $\Lhpi$, respectively. The symbol $\mathcal{P}_1(\cH)$ stands for the set of all rank-one projections on $\cH$. For any vectors $x,y\in\cH$ the operator $x\otimes y\in\Lh$ is defined by the formula $(x\otimes y)(z)=\langle z,y\rangle x\ (z\in\cH)$. Elementary computations show that the elements of $\mathcal{P}_1(\cH)$ are exactly the operators of the form $x\otimes x$ with some unit vector $x\in\cH$, moreover $\tr A(x\otimes x)=\langle Ax,x\rangle$ holds for each element $A\in\Lh$.
We shall also use the following extension of the usual logarithmic function:
\[
\hat{\log}(t)=\left\{
       \begin{array}{ll}
        \log(t), & t>0\\
        0, & t=0.
       \end{array}
     \right.
\]
The chaotic order on $\Lhpi$ is denoted by $\ll$, for any $A,B\in\Lhpi$ we have $A\ll B$ if and only $\log A\leq \log B$. A map $\Lambda: \Lh \to \Lh$ is said to be positive if $\Lambda(\Lhp)\subseteq \Lhp$ holds. Moreover, $\Lambda$ is called \emph{completely positive} if
the mapping
\[
\begin{pmatrix}
A_{11} & \quad A_{12} & \quad \ldots  & \quad A_{1n}\\
A_{21} & \quad A_{22} & \quad \ldots  & \quad A_{2n}\\
\vdots & \quad \vdots & \quad \ddots & \quad \vdots \\
A_{n1} & \quad A_{n2} & \quad \ldots & \quad A_{nn}
\end{pmatrix}
\mapsto
\begin{pmatrix}
\Lambda(A_{11}) & \quad \Lambda(A_{12}) & \quad \ldots  & \quad \Lambda(A_{1n})\\
\Lambda(A_{21}) & \quad \Lambda(A_{22}) & \quad \ldots  & \quad \Lambda(A_{2n})\\
\vdots & \quad \vdots  & \quad \ddots & \quad \vdots \\
\Lambda(A_{n1}) & \quad \Lambda(A_{n2}) & \quad \ldots  & \quad \Lambda(A_{nn})
\end{pmatrix}
\]
is positive for all $n\in\mathbb{N}$. Here, the matrix entries are elements of $\Lh$.

\subsection{Operator monotone and convex functions}

Let us recall here some basic facts from the theory of operator monotone and convex functions which will be needed. A real-valued function $f$ on an interval $J$ is called operator monotone if for any two Hermitian operators $A\leq B$ on a complex Hilbert space with spectra in $J$ we have $f(A)\leq f(B)$. Such functions have strong analytical properties, namely, according to the celebrated L\"owner theorem every operator monotone function defined on some open interval has an analytic continuation on the upper complex half plane. In particular, L\"owner's theorem implies immediately that every non-constant operator monotone function is strictly monotone increasing, as well.

A function $f:[0,\infty[\to \mathbb{R}$ is said to be operator convex if the operator inequality
\[
f(tA+(1-t)B)\leq tf(A)+(1-t)f(B),\qquad t\in[0,1]
\]
is valid for any positive operators $A,B$ on a complex Hilbert space. A function $f$ is termed to be operator concave, if $-f$ is operator convex. In fact, there is a deep connection between operator monotonicity and concavity on the nonnegative real line. It might be folklore for the reader that every operator monotone function is operator concave, and conversely, every operator concave function which is bounded from below is necessarily operator monotone.

Another important result is Kraus' theorem \cite{Kraus} which tells us that a function defined on the nonnegative real line is operator convex if and only if the first order divided difference function
\[
R(t,s)=\frac{f(t)-f(s)}{t-s}
\]
is operator monotone in both of its variables. Combining this fact with the famous theorem of L\"owner we conclude that every non-affine operator convex function is strictly convex. Furthermore, operator monotone and convex functions have certain integral representations. In particular, the following holds:
if $f:[0,+\infty[\to\mathbb{R}$ is an operator convex function satisfying $\omega(f):=\lim_{t \to +\infty}f(t)/t<+\infty$, then there exists a unique positive Radon measure $\nu_f$ such that
\begin{equation} \label{repr}
f(t)=f(0)+\omega(f)t-\int_{]0,+\infty[} \frac{t}{t+s} d\nu_f(s)
\end{equation}
and $\int_{]0,+\infty[} \frac{d\nu_f(s)}{1+s}<+\infty$.

\subsection{R\'enyi type relative entropies}

In the Introduction we have introduced three families of R\'enyi type divergences for positive definite operators. Next we give the precise definitions of these previously discussed quantities for arbitrary (not necessarily invertible) positive semidefinite operators. If $A=0$, then we simply set $D_\alpha(A\Vert B)=D_\alpha^{*}(A\Vert B)=D_{\alpha}^{\flat}(A\Vert B)=-\infty$. In the case where $A\neq 0$ the corresponding R\'enyi divergences are defined as follows.
\begin{itemize}
    \item[i.)] {
    If $\supp A \not\perp \supp B$ and ($0<\alpha<1$ or $\supp A \subseteq \supp B$), then
\[
D_{\alpha}(A\Vert B) = (\alpha-1)^{-1} \log \left((\tr A)^{-1}\tr A^{\alpha}B^{1-\alpha}\right)
\]
Otherwise, $D_{\alpha}(A\Vert B)=+\infty$.
    }
    \item[ii.)] {
    If $\supp A \not\perp \supp B$ and ($0<\alpha<1$ or $\supp A \subseteq \supp B$), then
\[
D_{\alpha}^{*}(A\Vert B) =(\alpha-1)^{-1} \log \left((\tr A)^{-1} \tr \left(B^{\frac{1-\alpha}{2\alpha}}AB^{\frac{1-\alpha}{2\alpha}}  \right)^{\alpha}\right)
\]
Otherwise, $D_{\alpha}^{*}(A\Vert B)=+\infty$.
    }
    \item[iii.)] {
    If $0<\alpha<1$ or $\supp A \subseteq \supp B$, then
\[
D_{\alpha}^{\flat}(A\Vert B) =(\alpha -1)^{-1} \log\left((\tr A)^{-1}\tr \left(P\exp\left(\alpha P\hat{\log}(A)P+(1-\alpha)P\hat{\log}(B)P\right)\right)\right),
\]
where $P$ denotes the projection on $\cH$ onto $\supp A\:\cap\:\supp B$.
Otherwise, if $\supp A \nsubseteq \supp B$ and $\alpha > 1$, then $D_{\alpha}^{\flat}(A\Vert B)=+\infty$.
    }
\end{itemize}

We remark that for commuting operators $A,B$ all the above divergences coincide.

\subsection{Different quantum $f$-divergences}

In what follows we turn to the notions of quantum $f$-divergences.
First of all, we note that both of them are non-commutative generalizations of  Csisz\'ar's $f$-divergences \cite{csiszar} which were originally defined between probability distributions $p=(p_1, p_2, \ldots, p_d)$ and $q=(q_1, q_2, \ldots, q_d)$ as
\[
S_f(p\Vert q)=\sum_{i=1}^{d}q_i f\left(\frac{p_i}{q_i} \right)\qquad q_i \neq 0,\quad i=1,2,\ldots, d
\]
where $f:[0,+\infty[\to \mathbb{R}$ is any strictly convex function.
To discuss the generalizations, for commuting operators $C,D\in\Lhpi$ we introduce the so-called perspective function $m_f:\Lhpi \times \Lhpi \to \Lhpi$ defined by $m_f(C,D)=Df(D^{-1}C)$.
Furthermore, we introduce the left and right multiplication operators acting on $\Lh$:
\[
L_A: X\mapsto AX,\qquad R_B: X\mapsto XB
\]
It is apparent that if the operators $A,B$ are positive, then so are $L_A, R_B$. Furthermore, the operators $L_A,R_B$ commute.
For a fixed operator $K\in\Lh$ the general notion of \emph{quasi-entropies} was defined in \cite{petz} by D. Petz (see also \cite{quasi}) as
\begin{align*}
S_f^K(A\Vert B) &=  \langle m_f(L_A,R_B)K,K \rangle_{HS}  \\
&=  \langle f(L_AR_{B^{-1}})KB^{1/2},KB^{1/2} \rangle_{HS}, \qquad A\in\Lhp,\ B\in\Lhpi
\end{align*}
where $\langle . , . \rangle_{HS}$ stands for the usual Hilbert-Schmidt inner product.
The quasi-entropies corresponding to the operator $K=I$ are termed as \emph{standard $f$-divergences} and denoted by $S_f(.\Vert .)$. Essentially following  \cite[2.1 Definition]{HMPB} the above notion can be extended for general $A,B\in\Lhp$ by the formula
\[
S_f(A\Vert B):=\lim_{\varepsilon \searrow 0} S_f(A\Vert B+\varepsilon I)=\sum_{a\in\sigma(A)}\left(\sum_{b\in\sigma(B)\backslash\{0\}}
bf\left(\frac{a}{b}\right)\tr P_aQ_b+\omega(f) a\tr P_aQ_0\right)
\]
where the operators $A,B$ are given by their spectral resolutions as $A=\sum_{a\in\sigma(A)}aP_a$ and $B=\sum_{b\in\sigma(B)}bQ_b$. We remark that according to \cite[Proposition 2.3]{HMPB} the limit in the last displayed formula indeed exists whenever $f$ is supposed to be continuous.

As for \emph{maximal $f$-divergences}, D. Petz and M. B. Ruskai defined these quantities for an operator convex function $f$ and operators $A\in\Lhp$ and $B\in\Lhpi$ by the formula
\[
D_f(A\Vert B)=\tr Bf(B^{-1/2}AB^{-1/2})
\]
as a generalization of the the usual Belavkin-Staszewski relative entropy which corresponds to the function $\eta(t)= t\log t$. However, K. Matsumoto introduced this sort of divergences in a more operational way. A \emph{reverse test} of pairs of positive definite operators $A,B\in\Lhpi$ is a triple $(p,q,\Gamma)$ where $p,q$ are finite probability distributions, and $\Gamma:\mathbb{C}^{d}\to \Lhpi$ is a trace preserving linear map\footnote{Here, trace preserving property means that $\tr \Gamma(p)= \sum_{i=1}^{d} p_i$ holds for any probability distribution $p$.} with the property that $\Gamma(p)=A$ and $\Gamma(q)=B$ hold. Then for any operator convex function $f$ we set
\[
D_f(A\Vert B)=\inf \{S_f(p \Vert q): \quad (p,q,\Gamma) \mbox{ is a reverse test}  \}.
\]
According to \cite[Lemma 9]{Matsumoto} the limit $\lim_{\varepsilon\searrow 0}D_f(A\Vert B+\varepsilon I)$ exists for any $A,B\in\Lhp$, whence this is a proper way to define maximal $f$-divergences for the non-invertible case, similarly to the standard ones.

Let us see in what sense are these $f$-divergences maximal. In \cite{Matsumoto} it was shown that the divergence $D_f(.\Vert .)$ is the maximal among the quantities $D_f^Q(.\Vert .)$ which
satisfy the following two conditions:
\begin{itemize}
    \item[i.)]{
    If $A$ and $B$ are commuting density operators with eigenvalues (counted with multiplicities) $a_1, a_2, \ldots, a_d$ and $b_1, b_2, \ldots, b_d$, respectively, then for the probability vector $p=(a_1, a_2, \ldots, a_d)$ and $q=(b_1, b_2, \ldots, b_d)$ we have $D_f^Q(A\Vert B)=S_f(p\Vert q)$. This property can be interpreted as these divergences distinguish at the highest level two different non-commuting positive operators.
    }
    \item[ii.)]{
    The quantites $D_f^Q(.\Vert .)$ are non-increasing under completely positive trace preserving (CPTP) maps, i.e.,
    $D_f^Q(\Lambda(A)\Vert \Lambda(B))\leq D_f^Q(A\Vert B)$ holds for any  CTPT map $\Lambda$.
    }
\end{itemize}
For further details about the properties of maximal $f$-divergences (at several points in comparison with the standard ones) the reader should consult with the articles \cite{HM},\cite{Matsumoto}.

\section{Main results}

Now we are in a position to formulate the main results of the paper. In order to do so, first observe that if $f$ is a continuous real valued function, then apparently $f(UAU^*)=Uf(A)U^*$ holds for any unitary operator $U$ acting on $\cH$ and each element $A\in\Lhs$ with $\sigma(A)\subseteq\text{dom}\ f$. We conclude that unitary congruences leave the divergence $D_{\alpha}^{\flat}(.\Vert .)$ invariant. One can verify easily that the transposition and scalar multiplications (with positive scalars) are also transformations on the set of positive definite operators which preserve this quantity. Our first theorem states that there are no other such transformations: all those maps which leave the divergence $D_{\alpha}^{\flat}(.\Vert .)$ invariant can be obtained from the previously mentioned types by composition. The precise formulation of our structural result follows.

\begin{theorem} \label{T:1}
Let $\alpha \in ]0,1[\cup]1,+\infty[$ be an arbitrary but fixed real number and let $\phi:\Lhpi\to\Lhpi$ be a bijection satisfying
\[
D_{\alpha}^{\flat}(\phi(A)\Vert\phi(B))=D_{\alpha}^{\flat}(A \Vert B),\qquad A,B\in\Lhpi.
\]
Then there exist either a unitary or an antiunitary operator $U$ on $\cH$ and a positive scalar $\lambda > 0$ such that
\begin{equation}\label{unitary}
\phi(A)=\lambda UAU^{*},\qquad A,B\in\Lhpi.
\end{equation}
\end{theorem}

As every antiunitary congruence transformation is a composition of the transposition and a unitary congruence transformation our result indeed yields the conclusion stated above.
We remark that the same result remains true also in the case where $\phi$ acts on the whole set $\Lhp$ instead $\Lhpi$. A sketchy proof of the statement is also given in Section 4.

In virtue of our first theorem a similar structural result might be expected concerning the symmetry transformations on $\Lhp$ with respect to maximal $f$-divergences. It turns out that the structure of the corresponding symmetries is much simpler: these maps are exactly the unitary-antiunitary congruence transformations.

\begin{theorem}\label{T:2}
Assume that $f:[0,+\infty[\to\RR$ is a non-affine operator convex function with the property $\lim_{t \to +\infty}f(t)/t < +\infty$. If $\phi:\Lhp\to\Lhp$ is a bijective map satisfying
\[
D_{f}(\phi(A)\Vert\phi(B))=D_{f}(A\Vert B),\qquad A,B\in\Lhp
\]
then there exists either a unitary or an antiunitary operator $U$ on $\cH$ such that
\[
\phi(A)= UAU^*,\qquad A\in\Lhp.
\]
\end{theorem}

We remark that the divergence corresponding to the affine function $f(x)=ax+b$ is $D_f(A\Vert B)=a\tr A+b\tr B$ which is not a proper measure of dissimilarity because it depends only on the trace of operators. Therefore, such quantities are ruled out of our investigations.

\section{Proofs}

This section is devoted to the proofs of our results formulated in the previous section. First, we present an auxiliary lemma on which we rely heavily in the sequel. We use the following notation in it. Let $N\colon\Lh\to\mathbb{R}$ be an arbitrary but fixed unitary invariant norm and $g\colon\mathbb{R}\to]0,\infty[$ be a strictly increasing bijection. Since the elements of $\mathcal{P}_1(\cH)$ are unitary similar to each other, we observe that the function $N(.)$ is constant on $\mathcal{P}_1(\cH)$. Let us denote this constant with $c_N$.

\begin{lemma}\label{L:1}
Let $A\in \Lhs$ and $x\in \cH$ be such that $\Vert x\Vert=1$. Then for the rank-one projection $P=x\otimes x$ we have
\[
\lim\limits_{t\to +\infty}N(g(A+t(P-I)))=c_N g(\langle Ax,x\rangle).
\]
\end{lemma}

\begin{proof}
The main step of the proof is to show that there exists a number $K>0$ such that for all scalars $t\ge K$ we have
\begin{equation}\label{E:est}
\begin{gathered}
\left(\langle Ax,x\rangle-t^{-\frac1{2}}\right)P+t(P-I)\le A\le\left(\langle Ax,x\rangle+t^{-\frac1{2}}\right)P+\frac t{2}(I-P).
\end{gathered}
\end{equation}
We verify only the first inequality, the second can be proved in a similar way. To do this, let $t>0$ be a number and define
\[
T(t):=A-\left(\left(\langle Ax,x\rangle-t^{-\frac1{2}}\right)P+t(P-I)\right).
\]
Fix an orthonormal basis in $\cH$ containing $x$. Then the off-diagonal entries of the matrix of $T(t)$ with respect to this basis are constants, the first entry of the diagonal is $t^{-1/2}$ and the $i$th one is of the form "$t$+constant" $(i=2,\ldots,n)$. It is apparent that the first leading principal minor $\Delta_1(t)$ of the matrix of $T(t)$ is positive. Now let $j\in\{2,\ldots,n\}$. It is easy to check that the $j$th leading principal minor $\Delta_j(t)$ of the matrix of $T(t)$ is a linear combination of powers of $t$ with maximal exponent $j-3/2$ and the coefficient of $t^{j-3/2}$ is 1. It follows that $\lim\limits_{t\to +\infty}\Delta_j(t)=+\infty$ and thus for large enough $t$ we have $\Delta_j(t)>0$. We infer that for large enough $t$ this inequality holds for all $j\in\{1,\ldots,n\}$ which, by Sylvester's criterion, implies the existence of a scalar $K>0$ satisfying that for all $t\ge K$ the matrix of $T(t)$ is positive definite. We easily conclude that the first inequality in \eqref{E:est} is valid for each $t\ge K$.

Next observe that if $S,T\in\Lhs$ are operators such that $S\le T$, then we necessarily have $N(g(S))\le N(g(T))$.
Indeed, in that case by Weyl's monotonicity theorem (c.f., e.g. \cite[Corollary III.2.3.]{B}) $\lambda_i\le\mu_i,$ where $\lambda_i$ and $\mu_i$ denote the $i$-th largest eigenvalue $(i=1,\ldots,d)$ of $S$ and $T$, respectively. Thus, as $g$ is increasing it follows that $g(\lambda_i)$ and $g(\mu_i)$ are the $i$th largest singular value of the positive operator $g(S)$ and $g(T)$, respectively, and $g(\lambda_i)\le g(\mu_i)\ (i=1,\ldots,d)$. Referring to \cite[Theorem 3.7]{SS}, we infer that $N(g(S))\le N(g(T))$. Having this in mind, by \eqref{E:est}, we get that for any $t\ge K$ we have
\begin{equation}\label{E:est1}
\begin{gathered}
N\left(g\left(\left(\langle Ax,x\rangle-t^{-\frac1{2}}\right)P+2t(P-I)\right)\right)\le N\left(g\left(A+t(P-I)\right)\right)\\
\le N\left(g\left(\left(\langle Ax,x\rangle+t^{-\frac1{2}}\right)P+\frac t{2}(P-I)\right)\right).
\end{gathered}
\end{equation}
We compute
\[
g\left(\left(\langle Ax,x\rangle-t^{-\frac1{2}}\right)P+2t(P-I)\right)=g\left(\langle Ax,x\rangle-t^{-\frac1{2}}\right)P+g(-2t)(I-P)
\]
and
\[
\begin{gathered}
g\left(\left(\langle Ax,x\rangle+t^{-\frac1{2}}\right)P+\frac t{2}(P-I)\right)
=g\left(\langle Ax,x\rangle+t^{-\frac1{2}}\right)P+g\left(-\frac t{2}\right)(I-P).
\end{gathered}
\]
By the properties of $g$, it is very easy to check that both of the latter expressions tend to $g(\langle Ax,x\rangle)P$ as $t$ tends to infinity. This gives us that the limit of both sides of \eqref{E:est1} at infinity is $$N(g(\langle Ax,x\rangle)P)=c_N g(\langle Ax,x\rangle),$$ whence taking the limit $t\to +\infty$ in \eqref{E:est1} yields the desired conclusion.
\end{proof}

We are now in a position to present the proof of our first main result.

\begin{proof}[Proof of Theorem~\ref{T:1}]

Let $\tilde{\exp}\colon[-\infty,\infty]\to[0,\infty]$ be the unique continuous extension of the function exp to $[-\infty,\infty]$. Observe that the preservation of the quantity $D_{\alpha}^{\flat}(.\Vert .)$ is equivalent to that of the quantity
\[
Q_{\alpha}^{\flat}(A\Vert B)=\tilde{\exp}((\alpha-1)D_{\alpha}^{\flat}(A\Vert B)), \qquad A,B\in \Lhp.
\]
Note that
\[
Q_{\alpha}^{\flat}(A\Vert B)=(\tr A)^{-1} \tr \exp(\alpha \log A+(1-\alpha)\log B), \qquad A,B\in \Lhpi.
\]
We intend to show now the validity of the following characterization of the chaotic order on $\Lhpi$ depending on the value of $\alpha$: if $0<\alpha <1$, then for any $C,B\in \Lhpi$ we have
\[
B\ll C \iff Q_{\alpha}^{\flat}(A\Vert B)\leq Q_{\alpha}^{\flat}(A\Vert C), \qquad A\in\Lhpi.
\]
Otherwise
\[
B\ll C \iff Q_{\alpha}^{\flat}(A\Vert B)\geq Q_{\alpha}^{\flat}(A\Vert C), \qquad A\in\Lhpi.
\]
We prove this only in the case where $0<\alpha <1$ (the case $\alpha > 1$ requires only very trivial modifications).

First assume that $B\ll C$, i.e., $\log B \leq \log C$ holds. Then for all $A\in\Lhpi$ we have $\alpha \log A + (1-\alpha) \log B \leq \alpha\log A + (1-\alpha) \log C $. The monotonicity of the function $X\mapsto \tr \exp(X)$ on $\Lhs$ implies that $\tr \exp\left(\alpha \log A + (1-\alpha) \log B\right) \leq \tr \exp\left(\alpha\log A + (1-\alpha) \log C\right) $. Dividing both sides of the inequality by $\tr A$ yields the necessity.

As for the sufficiency, plug $\exp(t(P-I))$ in the place of $A$ in the inequality displayed last but one, where $P=x\otimes x$ is any rank-one projection ($x\in\cH$). By the spectral mapping theorem, we see that the eigenvalues (counted with multiplicity) of the operator $\exp(t(P-I))$ are $1,e^{-t},\ldots,e^{-t}$ and thus
$\tr{\exp(t(P-I))}=1+(d-1)e^{-t}$. Then by Lemma~\ref{L:1} we compute
\[
\begin{gathered}
\lim_{t\to +\infty} Q_{\alpha}^{\flat}(\exp(t(P-I))\Vert B)= \\
\lim_{t\to +\infty} \frac{\tr \exp(\alpha t(P-I)+(1-\alpha) \log B))}{1+(d-1)e^{-t}}=\exp\left(\langle (1-\alpha) (\log B) x,x \rangle\right)
\end{gathered}
\]
and this remains true if we replace $B$ by $C$. Therefore for $A=\exp(t(P-I))$ the inequality in question implies that $\langle (\log B) x,x \rangle \leq \langle (\log C) x,x \rangle$. Since this holds for any unit vector $x\in\cH$, we infer that $\log B \leq \log C$, i.e., $B\ll C$, as we have desired.

By the characterization above, the bijective map $\phi$ is an order automorphism on $\Lhpi$ with respect to the chaotic order. The structure of such transformations is described in \cite{order}. It follows from \cite[Theorem 2]{order} that there exist an invertible either linear or conjugate linear operator $T$ and a Hermitian operator $H$ on $\cH$ such that $\phi$ is of the form
\begin{equation}\label{form}
\phi(A)=\exp(T(\log A)T^{*}+H),\qquad A\in\Lhpi.
\end{equation}

Next we show that the operator $H$ in \eqref{form} is a scalar multiple of the identity.
To verify this, referring to \eqref{form} and the preserver property of $\phi$, one has that
\begin{equation}\label{E:pres1}
\frac{\tr \exp((1-\alpha)T(\log B)T^*+H))}{\tr \exp H }=\frac{1}{d} \tr \exp((1-\alpha)(\log B)),\quad B\in \Lhpi.
\end{equation}
Let $C$ be any Hermitian operator. By plugging $\exp(T^{-1}C(T^*)^{-1})$ into the place of $B$ in \eqref{E:pres1}, we get that for the operator $S=T^{-1}$ one has
\[
\left(\frac{d}{\tr \exp H}\right)\tr \exp((1-\alpha)C+H)=\tr \exp((1-\alpha)SCS^*),\quad C\in \Lhs.
\]
Now let $x\in \cH$ be a unit vector, $t\in\mathbb{R}$ be an arbitrary positive or negative real number depending on whether $0<\alpha<1$ or $\alpha>1$. Set $P=x\otimes x$ and define $q=(1-\alpha)t$. By inserting $C=t(P-I)$ in the above displayed equality and applying Lemma~\ref{L:1}, we get that
\[
\begin{gathered}
\lim\limits_{q\to +\infty}\tr\exp(qS(P-I)S^*)=\left(\frac{d}{\tr \exp H}\right)\lim\limits_{q\to +\infty}\tr \exp(q(P-I)+H)\\
= \left(\frac{d}{\tr \exp H}\right)\exp(\langle Hx,x\rangle).
\end{gathered}
\]
It is apparent that the kernel of the positive operator $S(I-P)S^*$ is one-dimensional. It follows from the spectral mapping theorem that the eigenvalues of $\exp(qS(P-I)S^*)$ (counted with multiplicity) are $(1,\exp(q\mu_2),\ldots,\exp(q\mu_{d}))$ with some negative numbers $\mu_j<0$. Hence we have that
\[
\lim\limits_{q\to +\infty}\tr\exp(qS(P-I)S^*)=\lim\limits_{q\to +\infty}\left(1+\sum_{j=2}^{d}\exp(q\mu_j)\right)=1.
\]
By the last two displayed equalities $\exp(\langle Hx,x\rangle)=(\tr \exp H)/d$ holds for any unit vector $x\in\cH$, implying that the operator $H$ is a scalar multiple of the identity. In addition, we have
\begin{equation} \label{form2}
\phi(A)=\lambda\exp(T(\log A)T^*),\qquad A\in\Lhpi
\end{equation}
with some positive real number $\lambda>0$.

In what follows we show that the operator $T$ in \eqref{form2} is either unitary or antiunitary.
We assume that $T$ is linear (the case where $T$ is conjugate linear can be treated similarly). Consider the polar decomposition $T=U\left|T\right|$ of $T$ where $U$ is a unitary operator on $\cH$ and $\left|T\right|=\sqrt{T^*T}$ is a positive definite operator. Since
\[
\phi(A)=\exp(U\left|T\right|\log(A)\left|T\right|U^{*}+H)=U\exp(\left|T\right|\log(A)\left|T\right|+U^{*}HU)U^{*},\qquad A\in\Lhpi
\]
and the divergence $D_{\alpha}^{\flat}(.\Vert .)$ is invariant under unitary congruence transformations, without serious loss of generality we may and do assume that in \eqref{form2} we have $T\in\Lhpi$.

Using \eqref{form2} and the preserver property of $\phi$ we infer that
\begin{equation} \label{preserver}
\begin{split}
\frac{\tr \exp(\alpha T(\log A)T +(1-\alpha)T(\log B)T+(\log \lambda) I)}{\tr \lambda \exp(T(\log A)T)}=
\frac{\tr \exp(\alpha\log A +(1-\alpha)\log B)}{\tr A}
\end{split}
\end{equation}
holds for all $A,B\in\Lhpi$. Let $A=I$ and $B=tI$ with some real number $t>0$. We obtain from \eqref{preserver} that
\begin{equation} \label{plug}
\left(\frac{1}{d} \right)\tr \exp((1-\alpha)(\log t)T^2)=\left(\frac{1}{d}\right)\tr \exp((1-\alpha)(\log t)I )=t^{1-\alpha}.
\end{equation}
Suppose that $T\neq I$. Then $T$ has an eigenvalue $\mu$ which is different from $1$. Let $P_{\mu}$ be a rank-one projection on $\cH$ whose range is contained in the eigenspace of $T$ corresponding to $\mu$. Apparently, we have $\mu^2 P_{\mu}\leq T^2$.

First assume that $\mu > 1$ holds.
Then clearly for any $t>1$, we have
$$(1-\alpha)(\log t)\mu^2 P_{\mu}\leq (1-\alpha)(\log t)T^2$$ in the case where $0<\alpha<1$. Otherwise we deduce the above displayed inequality for $0<t<1$.
In both cases, by the monotonicity of the function $X\mapsto \tr \exp(X)$ on $\Lhs$, we have
\[
\tr \exp((1-\alpha)(\log t)\mu^2 P_{\mu}) \leq \tr \exp((1-\alpha)(\log t)T^2).
\]
Comparing this with \eqref{plug} yields
\[
\frac{\left(t^{1-\alpha}\right)^{\mu^2}+(d-1)}{d} \leq t^{1-\alpha}.
\]
Letting $t$ tend to infinity or zero depending on whether $0<\alpha<1$ or $\alpha>1$ in the above displayed inequality we obtain a contradiction because $\mu>1$. Consequently, all the eigenvalues of the operator $T$ are less than or equal to $1$.
Clearly, the transformation $\phi^{-1}(.)$ possesses the same preserver property as $\phi(.)$ and we have
\[
\phi^{-1}(A)=\exp\left(T^{-1}(\log((1/\lambda)A)T^{-1}\right).
\]
Therefore, in a similar fashion we get that all the eigenvalues are greater than or equal to $1$. The proof is complete.
\end{proof}

As we have already mentioned in the previous section, Theorem~\ref{T:1} remains valid in the case where $\phi$ acts not only on $\Lhpi$ but on the whole set $\Lhp$. To see this, we can proceed as follows.

\begin{proof}[Continuation of the proof when $\phi$ acts on $\Lhp$]

Our first aim is to show that the transformation $\phi$ preserves the rank of operators.
According to the value of $\alpha$ we divide the proof into two cases.

CASE I. We assume that $0<\alpha<1$.
In such a case we conclude that $\phi(0)=0$. Indeed, it follows from the following characterization of the element $0\in\Lhp$. Let $A\in\Lhp$. Then $A=0$ if and only if $Q_{\alpha}^{\flat}(A\Vert B)=\infty$ for every $B\in\Lhp$.

Clearly, for every non-zero $A,B\in\Lhp$ we have $Q_{\alpha}^{\flat}(A \Vert B)=0$ if and only if $\supp{A}\cap\supp{B}=\{0\}$.
Therefore, the assumption in the theorem implies that for any $A,B\in\Lhp$ we have
\[
\supp{A}\cap\supp{B}=\{0\} \iff \supp{\phi(A)}\cap\supp{\phi(B)}=\{0\}.
\]
One can easily verify that this yields
\begin{equation} \label{inclusion}
\supp A \subseteq \supp B \iff \supp \phi(A) \subseteq \supp \phi(B).
\end{equation}
From \eqref{inclusion} we infer that $\phi$ preserves the rank of the elements of $\Lhp$.

CASE II.
If $\alpha>1$, then \eqref{inclusion} follows immediately from the property that $Q_{\alpha}^{\flat}(A\Vert B)$ is finite if and only if $\supp A \subseteq \supp B$.

As a corollary of the above, we conclude that the transformation $\phi$ maps the set $\Lhpi$ into itself.
An application of Theorem~\ref{T:1} gives us that there is either a unitary or an antiunitary operator $U$ on $\cH$ and a positive scalar $\lambda$ such that $\phi(A)=\lambda UAU^*$ holds for any $A\in\Lhpi$.
In what follows we determine $\phi$ on the whole $\Lhp$.
Considering the transformation $\psi(.)=(1/\lambda)U^*\phi(.)U$ we obtain a bijective map which preserves the divergence $D_{\alpha}^{\flat}(. \Vert .)$ and has the additional property that it acts as the identity on $\Lhpi$. In particular, we have $\psi(I)=I$.
Let $P$ be any rank-one projection. As $\phi$ preserves the rank-one operators, so is $\psi$.
Clearly, for any such operator $R$ we have $D_{\alpha}^{\flat}(R\Vert I)=0$ if and only if $R\in\Ph$. Indeed, if $R=z\otimes z\in\Ph$, then we have $D_{\alpha}^{\flat}(R\Vert I)= - \langle (\log I)z,z \rangle =0$. Conversely, the condition
$D_{\alpha}^{\flat}(R\Vert I)=0$ is equivalent to
\[
Q_{\alpha}(R\Vert I)=\frac{\tr \left( (R/\|R\|)\exp(\alpha \hat{\log}(R))\right)}{\tr R}=\frac{\|R\|^{\alpha}}{\|R\|}=1
\]
implying that $\|R\|=1$. As $\psi(I)=I$,
we infer that if $P$ is a rank-one projection, then so is $\psi(P)$.

Next choose unit vectors $x$ and $y$ from the ranges of $P$ and $\psi(P)$, respectively. For any $A\in\Lhpi$, we have
\[
-\langle (\log A) x, x \rangle=D_{\alpha}^{\flat}(P \Vert A)= D_{\alpha}^{\flat}(\psi(P) \Vert A) = -\langle (\log A) y, y \rangle
\]
implying that the vectors $x,y$ are scalar multiples of each other. It means that $\psi(P)=P$, whence $\psi$ acts as the identity on the set of rank-one projections as well. This gives us that $\psi$ preserves the supports of the operators in $\Lhp$. If $\supp P \subseteq \supp A = \supp \psi(A)$, then we obtain that for any unit vector $z\in \supp P$
\[
-\langle (\hat{\log}(A)) z, z \rangle=D_{\alpha}^{\flat}(R \Vert A)= D_{\alpha}^{\flat}(R \Vert \psi(A)) = -\langle (\hat{\log}(\psi(A))) z, z \rangle.
\]
It follows that $\psi(A)=A$ holds on $\supp A$ and thus $\psi$ is the identity transformation on the whole $\Lhp$. This completes the proof.
\end{proof}

Now we are in a position to verify our second theorem. We remark that in the argument we borrow some ideas from \cite{virosztek}.

\begin{proof}[Proof of Theorem~\ref{T:2}]

Let us define the function $h_f:[0,+\infty[\to \mathbb{R}$ by the formula
\[
h_f(t)=\int_{]0,+\infty[} t/(t+s) d\nu_f(s)\quad(t\in[0,\infty[)
\]
where $\nu_f$ is the Radon measure associated with the operator convex function $f$ appearing in \eqref{repr}. Observe that $h_f$ is a non-negative operator monotone function and by \eqref{repr} for any operators $A,B\in\Lhp$ we have the following equality:
\begin{equation} \label{Drep}
\begin{gathered}
D_f(A\Vert B)= f(0)\tr B + \omega(f)\tr A -\\
\lim_{\varepsilon\searrow 0}\tr (B+\varepsilon I)^{1/2}h_f((B+\varepsilon I)^{-1/2}A(B+\varepsilon I)^{-1/2})(B+\varepsilon I)^{1/2}.
\end{gathered}
\end{equation}
By the theory of Kubo-Ando means \cite{KuAn}, for any real number $\varepsilon>0$  the operator
\[
(B+\varepsilon I)^{1/2}h_f\left((B+\varepsilon I)^{-1/2}A(B+\varepsilon I)^{-1/2}\right)(B+\varepsilon I)^{1/2}
\]
is the mean $\sigma_{h_f}$ of $B+\varepsilon I$ and $A$ associated with the function $h_f$. The continuity property of such means tells us that for any sequences $A_n\downarrow A$ and $B_n\downarrow B$ in $\Lhp$ we have $A_n\sigma_{h_f} B_n\downarrow A\sigma_{h_f}B$ where $\downarrow$ refers to the convergence of a monotone decreasing sequence in the strong operator topology. The previous observations yield that
\begin{equation} \label{Drepm}
D_f(A\Vert B)= f(0)\tr B + \omega(f)\tr A -\tr B\sigma_{h_f}A,\quad A,B\in\Lhp.
\end{equation}
We aim to prove that $\phi$ preserves the trace. Let us see how this results in the desired conclusion. By the last displayed equality, the trace preserving property implies that $\phi$ leaves the quantity $\tr A\sigma_{h_f}B$ invariant. Then it is straightforward to see that $\phi$ fulfills the conditions of \cite[Theorem 2]{integralequ} and the application of that result yields the statement of Theorem~\ref{T:2}.
Apparently, if $f(1)\neq 0$, then we have $$f(1)\tr \phi(A)=D_f(\phi(A)\Vert\phi(A))=D_f(A\Vert A)=f(1)\tr A$$ implying that $\phi$ leaves the trace invariant. Therefore, below we suppose that $f(1)=0$.

Next we prove the following characterization of the element $0\in\Lhp$.
Let $A\in\Lhp$. Then $A=0$ if and only if
\[
D_f(B\Vert C)\leq D_f(B\Vert A)+D_f(A\Vert C),\qquad B,C\in\Lhp.
\]
Since the quantity $\tr B\sigma_{h_f}A$ is non-negative, we have $D_f(B\Vert C)\leq f(0)\tr C + \omega(f)\tr B$, by \eqref{Drepm}. Furthermore, $D_f(B\Vert 0)=\omega(f)\tr B$ and $D_f(0\Vert C)=f(0)\tr C$. This shows $$D_f(B\Vert C)\leq D_f(B\Vert 0)+D_f(0\Vert C).$$
To see the converse, assume that $A\neq 0$ and thus $\tr A\neq 0$.
As $f$ is a non-affine operator convex, and hence a strictly convex function on the interval $]0,+\infty[$ we obtain that the first order divided difference function
\[
R(t,s):=\frac{f(t)-f(s)}{t-s}
\]
is strictly monotone increasing in both of its variables. Therefore, for any $t,s>1$ we have
\[
\frac{f(t)-f(1)}{t-1}<\frac{f(st)-f(s)}{s(t-1)}
\]
or, equivalenty, since $f(1)=0$ one has
\[
sf(t)+f(s)<f(ts).
\]
Multiplying by $\tr A$ we obtain
\[
f(t)\tr(sA)+f(s)\tr A<f(ts)\tr A.
\]
The divergence $D_f(.\Vert .)$ is positively homogeneous, whence the above displayed inequality is equivalent to
\[
D_f(tsA\Vert sA)+D_f(sA\Vert A)<D_f(tsA\Vert A)
\]
which verifies our claim.

Since $\phi(0)=0$ we deduce from \eqref{Drep} that
\[
f(0)\tr \phi(A)=D_f(0\Vert \phi(A))=D_f(0\Vert A)=f(0)\tr A
\]
implying that $\phi$ is trace preserving whenever $f(0)\neq 0$.

Finally, if $f(0)=f(1)=0$, then $f$ necessarily has a minimum $c<0$ because of strict convexity. In such a case we can argue just as in the corresponding part in \cite[Theorem 2.3]{banach}. We assert that
\begin{equation} \label{inf}
\inf_{A\in\mathcal{L}(\mathcal{H})_+} D_f(A\Vert B) = c \tr B, \qquad B\in\mathcal{L}(\mathcal{H})_+
\end{equation}
holds. Indeed, one has
\[
D_f(tB\Vert B) = f(t) \tr B,  \qquad B\in\Lhp,\ t\in[0,\infty[
\]
and by taking minimum in $t$ here, we obtain that $\inf_{A\in\Lhp} D_f(A\Vert B) \leq c \tr B$.

To see the reverse inequality, if $B=0$, then the equation $D_f(A\Vert B)=0=c\tr B$ must hold, since $\tr B =0$. If $B\neq 0$, then we calculate
\[
\begin{gathered}
 \frac{D_f(A\Vert B+\varepsilon I)}{\tr (B+ \varepsilon I)} = \tr\left( \frac{B+\varepsilon I}{\tr (B+\varepsilon I)} f\left( {(B+\varepsilon I)}^{-1/2} A {(B+\varepsilon I)}^{-1/2} \right) \right) \geq \\   f\left( \frac{\tr \left( {(B+\varepsilon I)}^{1/2} A {(B+\varepsilon I)}^{-1/2}\right)}{\tr (B+\varepsilon I)} \right)\geq c
\end{gathered}
\]
where the first inequality follows from \cite[A.2 Lemma]{HMPB} which claims that for every Hermitian operator $H$ and density operator $D$ acting on a finite dimensional Hilbert space and for any convex function $f$ we necessarily have $\tr D f(H) \geq f\left(\tr DH\right)$. Therefore, we arrive at the inequality
\[
D_f(A\Vert B)=\lim_{\varepsilon \searrow 0} D_f(A\Vert B+\varepsilon I)\geq \lim_{\varepsilon \searrow 0} c \tr(B+\varepsilon I)=c \tr B
\]
from which we deduce that \eqref{inf} holds for every $B\in\Lhp$. Hence the preserver property and the bijectivity of $\phi$ imply
\[
\begin{gathered}
c\tr B = \inf_{A\in\Lhp} D_f(A\Vert B) = \inf_{A\in\Lhp} D_f(\phi(A)\Vert \phi(B)) =\\
\inf_{A\in\Lhp} D_f(A\Vert \phi(B)) = c\tr \phi(B).
\end{gathered}
\]
Since $c\neq 0$ the above displayed formula yields that $\phi$ is trace preserving. The proof is complete.
\end{proof}

\section{Remarks}

Related to Theorems~\ref{T:1} and \ref{T:2} a natural question arises: what happens when $\phi$ acts on the set of density operators instead of the cones of either positive semidefinite or positive definite operators. It turns out that the machinery developed in \cite{MNSz} (see also \cite{banach},\cite{gaal}) based on the non-bijective version of Wigner's theorem concerning the most fundamental symmetry transformations (i.e. the maps preserving the transition probability between pure states) can be applied to achieve the following result: if $\phi$ is an arbitrary (no bijetivity is assumed) transformation on the set of density operators which has one of the preserver properties formulated in those theorems, then it is necessarily implemented by either a unitary or antiunitary similarity transformation.
We have decided not to formulate the corresponding results here.

Nevertheless, another interesting and exciting question arises. One can observe that in the case $\alpha=1/2$, the algebraic operation $$A\diamond B= \exp\left(\frac{\log A + \log B}{2}\right)$$ appearing in the quantity $D_{\alpha}^{\flat}(.\Vert.)$ is the so-called Log-Euclidean mean. In fact, this algebraic operation can be regarded as the \textit{relative operator entropy} version of the divergence $D_{1/2}^{\flat}(.\Vert.)$. It can be extended to the set of arbitrary (not necessarily invertible) positive semidefinite operators by the definition
\[
A\diamond B= P_{A,B} \exp\left( \frac{P_{A,B}(\hat{\log}A)P_{A,B} +P_{A,B}(\hat{\log} B)P_{A,B}}{2}\right)P_{A,B}\quad(A,B \in\Lhp)
\]
where $P_{A,B}$ denotes the projection on $\cH$ onto $\supp A \cap \supp B$.
One may ask how we can describe the corresponding isomorphisms with respect to this operation. Our last theorem below gives us the answer under a mild regularity assumption (i.e., under a sort of continuity).

To formulate the result we need the concept of the logarithmic product.
This binary operation is given by $A \odot B=\exp(\log A+\log B)$ for any $A,B\in\Lhpi$. According to \cite{dolinar} it can be extended using the famous Lie-Trotter formula by the definition $A\odot B=\lim_{n\to\infty}\left(A^{1/n}B^{1/n}\right)^{n}$. We have the following explicit and very useful formula:
\[
A\odot B = P_{A,B} \exp(P_{A,B}(\hat{\log}A)P_{A,B} + P_{A,B}(\hat{\log} B)P_{A,B})P_{A,B}.
\]

Our last result reads as follows.

\begin{theorem}\label{T:3}
Let $\phi:\Lhp\to\Lhp$ be a bijective map satisfying
\[
\phi(A\diamond B)=\phi(A)\diamond\phi(B),\qquad A,B\in\Lhp.
\]
Assume that $\phi$ is continuous at $0$ with respect to the topology on $\Lhp$ inherited from the unique Hausdorff linear topology on $\Lhs$.
Then there exists an invertible either linear or conjugate linear operator $T$ on $\cH$ such that
\[
\phi(A)=\phi(I)\odot P_{T,A}\exp\left(P_{T,A} T(\hat{\log}A)T^* P_{T,A}\right)P_{T,A}
\]
where $P_{T,A}$ denotes the orthogonal projection on $\cH$ onto $\left(T(\ker A)\right)^{\perp}$.
\end{theorem}

For the proof we need the following lemma.

\begin{lemma}\label{L:5}
Let $\phi:\Lhp\to\Lhp$ be a bijective map satisfying
\begin{equation}\label{mean}
\phi(A\diamond B)=\phi(A)\diamond\phi(B),\qquad A,B\in\Lhp.
\end{equation}
Then for a given $X\in\Lhp$ the transformation $A\mapsto X\odot\phi(A)$ also satisfies \eqref{mean}.
\end{lemma}
\begin{proof}
Define $\psi(A)=X\odot\phi(A)$. Since $C\odot C = C^2$, we have
\[
(\sqrt{A}\odot \sqrt{B})^2=\sqrt{A}\odot \sqrt{B}\odot\sqrt{A}\odot\sqrt{B}=A\odot B
\]
implying that $\sqrt{A}\odot \sqrt{B}=\sqrt{A\odot B}$.
Hence we compute
\[
\begin{gathered}
\psi(A\diamond B)=X\odot\phi(A\diamond B)=X\odot(\phi(A)\diamond\phi(B))=\\
X\odot \sqrt{\phi(A)}\odot\sqrt{\phi(B)}=\sqrt{X}\odot \sqrt{\phi(A)}\odot\sqrt{X}\odot \sqrt{\phi(B)}=\\
\sqrt{X\odot\phi(A)}\odot\sqrt{X\odot\phi(B)}=\psi(A)\diamond\psi(B)
\end{gathered}
\]
which yields the desired conclusion.
\end{proof}

We finish the paper with the proof of our last result.

\begin{proof}[Proof of Theorem~\ref{T:3}]

First, we claim that $\phi$ preserves the invertible elements. Indeed, it follows from the following characterization of such elements: the operator $A\in\Lhp$ is nonsingular if and only if for any $B\in\Lhp$ the equation
$A\diamond X = B$ has a solution $X\in\Lhp$. In particular, we have that $\phi(I)$ is invertible.

Let us define the transformation $\psi\colon\Lhp\to\Lhp$ by $\psi(A)=\phi(I)^{-1}\odot\phi(A)$. According to Lemma~\ref{L:5} the map $\psi$ preserves the Log-Euclidean mean. As $C\odot C^{-1}=I\ (C\in\Lhpi)$, the transformation $\psi$ has the additional property that it sends the identity into itself.
Therefore, we have
$\psi(\sqrt{A})=\psi(A\diamond I)=\psi(A)\diamond I=\sqrt{\psi(A)}$
meaning that $\psi$ is compatible with the square root operation. This easily yields that $\psi$ preserves the logarithmic product. Hence an application of \cite[Theorem]{dolinar} gives us that there exists an invertible either linear or conjugate linear operator $T$ on $\cH$ such that
\begin{equation}\label{psi}
\psi(A)=\phi(I)^{-1}\odot \phi(A)= P_{T,A}\exp\left(P_{T,A} T(\hat{\log}A)T^* P_{T,A}\right)P_{T,A}
\end{equation}
implying that $\phi$ is of the required form.
\end{proof}

\section{Acknowledgements}
The authors were supported by the "Lend\"ulet" Program (LP2012-46/2012) of the Hungarian Academy of Sciences and by the National Research, Development and Innovation Office -- NKFIH Reg. No.

\end{document}